\documentclass{amsart}

\usepackage{amssymb}

\newtheorem{thm}{Theorem}[section]

\newtheorem{lem}[thm]{Lemma}
\newtheorem{cor}[thm]{Corollary}
\theoremstyle{definition}

\theoremstyle{remark}
\newtheorem{remark}[thm]{Remark}
\numberwithin{equation}{section}

\begin{document}

\title{Positivity and Green's operators}

\author{David Raske}

\email{nonlinear.problem.solver@gmail.com}

\begin{abstract}
It is well known that positive Green's operators are not necessarily positivity preserving. In this paper we investigate the matter of just how far from being positivity preserving a positive Green's operator can be. We will also identify a broad class of Green's operators that are not necessarily positivity preserving but have properties related to positivity preservation that one expects from positivity preserving Green's operators.
\end{abstract}

%section{introduction}
%It iis well known that in general Green's operators associated with positive elliptic differential operators and Dirichlet boundary %conditions will not be positivity preserving. In this paper we investigate the problem of how far can it be from being positivity %preserving. We will also state and prove the existence of non-positive preserving Green's operators that still possess properties that t%ypically one expects from positivity preserving Green's operators.

\maketitle

\section{Introduction}

Let $n$ be a positive integer. Let $U$ be a bounded domain of $\mathbb{R}^n$ whose boundary $S$ consists of a finite number of smooth hypersurfaces. Consider the equation
\begin{equation}\label{eq 1}
\begin{aligned}
& A(x,\partial/\partial x)u(x)=f(x) \quad x \in U, \\
&B_j(x, \partial/\partial x)u(x)=0, \quad x \in S, \quad j=1,2,...\frac{m}{2}
\end{aligned}
\end{equation}
where $A$ is an elliptic operator of order $m$ and the boundary operators $\{B_j\}$ satisfy: (i) At every point $x$ of $S$, the normal derivative is not characteristic for any $B_j$; and (ii) the order $m_j$ of  $B_j$ is less than $m$, and $m_j \ne  m_k$. The domain $\mathcal{D}(A)$ of $A$ is defined by
\begin{equation}\label{eq 2}
\begin{aligned}
\mathcal{D}(A) = \{ & u(x) | u \in H^m(U) \quad \text{and} \\
&B_j(x,\partial/\partial x)u(x)=0 \quad \text{for} \quad x \in S,\\
&  j=1,2,...,\frac{m}{2}\}.
\end {aligned}
\end{equation}
We will call the ordered collection 
$$
(U, A, B_1,...,B_{m/2})
$$ 
admissable if and only if (i) $U$, $A$ and $\{B_j\}$ are defined as above, (ii) A is a one-to-one mapping from $\mathcal{D}(A)$ onto $L^2(U)$; and (iii) $A$ has the property that $A^{-1}$ is self-adjoint in $L^2(U)$.  The positive integer $m$ will always be the order of $A$, and the positive integer $n$ will always be the dimension of $\mathbb{R}^n$. We will call the inverse $G = A^{-1}$ a Green's operator. 

Green's operators received a considerable amount of attention in the twentieth century. In \cite{Schechter} it was shown that $G[u] \in H^m(U)$ if $u(x) \in L^2(U)$ and that $G[u]$ depends continuously on $u$. In particular, if $m > n/2$, then by Sobolev's theorem $G$ is a continuous mapping from $L^2(U)$ onto $C^0(\bar{U})$. In \cite{Garding} it was shown that $G$  is represented by an integral operator of Hilbert-Schmidt type. Namely, for any $f(x) \in L_2(U),$
\begin{equation}\label{Fubini}
\begin{aligned}
& (Gf)(x) = \int_U G(x,\xi) f(\xi) \, d\xi,  \\
&\int \int  |G(x,\xi)|^2 \,dx \, d\xi < +\infty.
\end{aligned}
\end{equation}
In general, the function $G(x,\xi)$, obtained by the kernel representation of the Green's operator $G$, is called a Green's function.

\section{higher order boundary value problems}

For second-order elliptic differential operators there is a nice co-incedence: if $A$ is a positive operator, and if $A^{-1}$ is self-adjoint in $L^2(U)$, then $A^{-1}$ is positivity preserving if homogeneous Dirichlet boundary conditions are used. Given this, it is  natural to ask whether or not positive Green's operators \footnote{Here, and for the remainder of the paper, we will say that a Green's operator is positive if and only if all of its eigenvalues are real and positive.} associated with higher-order elliptic differential operators and homogeneous  Dirichlet boundary conditions are positivity preserving. There are many counter examples to this conjecture, although some  Green's operators associated with higher-order elliptic differential operators and homogeneous Dirichlet boundary conditions are positivity preserving. One can still ask if the Green's operator has properties that are related to the positivity preserving property. Examples include whether or not the principal eigenvalue of $A$ is simple and whether or not the corresponding eigenfunction is of one sign. These questions have received quite a bit of attention. Again, some admisssable ordered collections $(U,A,B_1,...,B_{m/2})$  have these properties, some don't. (See \cite{GGS} and references therein for more on the history of these problems.) Equally surprising is the fact that there exists admissable ordered collections where $A^{-1}$ is positive and such that the solution of
\begin{equation}\label{BVP}
\begin{aligned}
& A(x,\partial/\partial x)u(x)=1 \quad x \in U, \\
&B_j(x,\partial/\partial x)u(x)=0, \quad x \in S, \quad j=1,2,...,\frac{m}{2}
\end{aligned}
\end{equation}
 is not non-negative. (See \cite{GS}.) Later in this paper we will see that this implies that it is possible that the mean value of a solution of 
\begin{equation}
\begin{aligned}
& A(x,\partial/\partial x)u(x)=f(x) \quad x \in U, \\
&B_j(x, \partial/\partial x)u(x)=0, \quad x \in S, \quad j=1,2,...,\frac{m}{2},
\end{aligned}
\end{equation}
where $f$ is a positive function, isn't necessarily non-negative. This, in turn, implies that there exists situations where the Green's operator $G$ is a positive operator, but we do not have $\int_U G(x,\xi) \, d\xi \geq 0$ for all $x$. 

One statement involving positivity that is true for all Green's operators associated with positive elliptic operators is as follows:

%\begin{thm} Let $A$ be a positive elliptic differential  operator defined in \eqref{eq 1} and \eqref{eq 2}. Suppose it has a Green's %operator $G$ and a Green's function $G(x,\xi)$. Then we have it that $\int \int w(x) G(x,\xi) w(\xi) \, dx \, d\xi$ for all $w: D \%rightarrow \mathbb{R}$. \end{thm}

\begin{thm}
Suppose that the ordered collection
$$
(U, A, B_1,...,B_{m/2})
$$ 
is admissable, $A$ is positive, and $m > \frac{n}{2}$. Then we have 
 \begin{equation}\label{eq 3}
\int \int z(x) G(x,\xi)  z(\xi) \, dx \, d\xi \geq 0
\end{equation}
for all $z \in L^2(U)$, where $G(x,\xi)$ is the Green's function associated with the admissable ordered collection $$(U,A,B_1,...,B_{m/2}).$$
\end{thm}

\begin{proof} First note that we have $\int_U w(x) A(x,\partial/\partial x) w(x) \, dx \geq 0$, for all $w$ in the domain of $A$ because we assume that $A > 0$. The positivity of $A$ also lets us conclude that a Green's operator $G$ exists, so let us write $w = G[z]$. Then we have it that $\int_U z(x)G[z](x) \, dx \geq 0$ for all $z \in L_2(U)$. Since $G(x,\xi)$ is obtained by the kernel representation of the operator $G$, we can write $\int \int z(x) G(x,\xi)  z(\xi) \, dx \, d\xi \geq 0$ for all $z \in L^2(U).$ This gives us the theorem.
\end{proof}

Some consequences of the above theorem are as follows:

\begin{cor} Suppose that the ordered collection
$$
(U, A, B_1,...,B_{m/2})
$$ 
is admissable, $A$ is positive, and $m > \frac{n}{2}$. Then we have that 
$$
\int \int G(x,\xi)  \, dx \, d\xi \geq 0
$$
\end{cor}

\begin{proof} Recall \eqref{eq 3} and set $z = 1$. 
\end{proof}

\begin{remark}
Let $(U,A, B_1,...,B_{m/2})$ be an admissable ordered collection. Note that the above corollary implies that the mean value of the solution of \eqref{BVP} is positive if $A$ is a positive operator..
\end{remark}

\begin{cor} Suppose that the ordered collection
$$
(U, A, B_1,...,B_{m/2})
$$ 
is admissable, $A$ is positive, and $m > \frac{n}{2}$. then the solution of \eqref{eq 1} is positive somewhere on $U$ if $f$ is positive.
\end{cor}

\begin{proof}
Note that our hypotheses allow us to write 
$$
0 < \int_U (\int_U f(x)  G(x,\xi)  f(\xi)\, \, d\xi )\, dx= \int_U f(x) (\int_U G(x,\xi) f(\xi)  \, d\xi) \, dx.
$$ 
The corollary follows.
\end{proof}

\section{Equivalence of three properties}

It is difficult to say much more about positivity preservation for arbitrary Green's operators, without adding additional hypotheses.
Given that we want $G(x,y) \geq 0$, but only have $\int \int G(x,y) \, dx \, dy \geq 0$ in general, it seems natural to ask questions about the case where $\int_U G(x,y) \, dy \geq 0$ for all $x \in U$. We reserve the remainder of the paper to investigating this property. First we will prove two lemmas relating the positivity of $\int_U G(x,y) \, dy$ to two different but important properties. 

\begin{lem} Let $(U,A, B_1,...,B_{m/2})$ be an admissable ordered collection. Suppose $m > n/2$, and let $G(x,y)$ be the Green's function associated with this admissable ordered collection. Let $f: U \rightarrow \mathbb{R}$ be a smooth, non-negative function and let the solution of \eqref{eq 1} be called $u_f$. Then we have $\int_U u_f(x) \, dx \geq 0$ iff $\int_U G(x,y) \, dy \geq 0$ for all $x \in U$. 
\end{lem}

\begin{proof}

First note that we have

\begin{equation}
\int_U u_{f}(x) \,dx = \int_{U} ( \int_U  f(y) G(x,y) \, dy )\,  dx.
\end{equation}

Invoking the symmetry of $G(x,y)$, we can write
\begin{equation} 
\int_U u_{f}(x) \, dx = \int_{U} (\int_U f(y) G(y,x) \, dy) \, dx.
\end{equation}

We can now use Fubini's Theorem to obtain

\begin{equation}
\int_U u_{f}(x) \, dx = \int_{U} ( \int_U  f(y) G(y,x) \, dx) \, dy.
\end{equation}

This, in turn, allows us to write

\begin{equation}\label{eq}
\int_U u_{f}(x) \, dx = \int_{U} f(y) ( \int_U  G(y,x) \, dx) \, dy.
\end{equation}

Now, suppose that $\int_U G(x,y) \, dy \geq 0$ for all $x \in U$. This, in turn, allows us to conclude that $\int_U G(y,x) \, dx \geq 0$ for all $y \in U$.  Since $f$ is non-negative, we can now conclude that the right hand side of \eqref{eq} is non-negative. This, in turn, allows us to write $\int_U u_f(x) \, dx \geq 0$.

We now have $\int_U u_f(x)  \,  dx \geq 0$ if $\int_U G(x,y) \, dy \geq 0$ for all $x \in U$. It remains to show that there exists a smooth, non-negative function $f: U \rightarrow \mathbb{R}$ such that  $\int_U u_f(x)  \,  dx <  0$ if 
\begin{equation}\label{hypothesis}
\int_U G(x,y) \, dy < 0 \quad \text{for some} \quad x \in U.
\end{equation}
To see that this is the case, we first note that \eqref{hypothesis} implies that for some $y \in U$, $\int_U G(y,x) \, dx < 0$. Continuity then allows us to conclude that there exists an open ball $B$ in $U$ such that $\int_U G(y,x) \, dx < 0$ for all $y \in B$.   This,  in turn, allows us to find a smooth, non-negative function $f: U \rightarrow \mathbb{R}$ such that $\int_U f(y) (\int_U G(y,x) \, dx) \, dy < 0$. The lemma follows.
\end{proof}

Next we prove

\begin{lem} Let $(U,A, B_1,...,B_{m/2})$ be an admissable ordered collection. Suppose $m > n/2$, and let $G(x,y)$ be the Green's function associated with this admissable ordered collection. Then $\int_U G(x,y) \, dy \geq 0$ for all $x \in U$ iff the unique solution of the boundary value problem \eqref{BVP} is non-negative.
\end{lem}

\begin{proof} Suppose that the solution, $u$, of the boundary value problem \eqref{BVP} is non-negative on $U$. Note that our hypotheses allow us to assume that a Green's operator $G$ exists and that $u = G[1]$. Our hypotheses also guarantee that a Green's function $G(x,y)$ associated with $G$ exists, and we can write $u(x) = \int_U G(x,y) \, dy$. Recall that one of our hypotheses is that the solution of \eqref{BVP} is non-negative, so we have it that $\int_U G(x,y) \, dy \geq 0$ for all $x \in U$, if the unique solution of the boundary value problem \eqref{BVP} is non-negative .
 
Now, let us assume that $\int_U G(x,y) \, dy \geq 0$ for all $x \in U$, and let us write
$$
w(x) = \int_U G(x,y) \, dy.
$$
Notice that $w \in \mathcal{D}(A)$. Applying $A$ to both sides of the above equation gives us 
$$
A(x,\partial/\partial x) w(x) = \int_U A(x,\partial/\partial x)G(x,y) \,  dy = \int_U \delta(x-y) \, dy = 1.
$$
Since $w \in \mathcal{D}(A)$, we have it that $w$ is the solution of \eqref{BVP}. The lemma follows.
\end{proof}

With the above two lemmas in place, we can conclude that

\begin{thm} Let $m$ and $n$ be positive integers, and suppose that $m > n/2$. Let $(U,A,B_1,...,B_{m/2})$ be an admissable ordered collection. Then the following statements are equivalent to each other. (i) The solution of \eqref{BVP} is non-negative; (ii) $\int_U G(x,y) \, dy \geq 0$ for all $x \in U$, where $G(x,y)$ is the Green's function associated with the admissable ordered collection $(U,A, B_1,...,B_{m/2})$; and (iii) the mean value of the solution of \eqref{eq 1} is non-negative if $f$ is non-negative. \end{thm}

\begin{proof} Combine Lemma 3.1 and Lemma 3.2. \end{proof}

\begin{remark} Recall that there exists admissable ordered collections 
$$
(U,A,B_1,...,B_{m/2})
$$ 
with $m > n/2$ such that the solution of \eqref{BVP} is negative somewhere. Hence, the above theorem might be useful as a classification tool.
\end{remark}

\begin{remark} There is a whole separate highly developed theory concerning positivity preservation and higher order boundary value problems. (See \cite {GFS} for an example of how this theory works.) It produces bounds on the negative part of a Green's function associated with a higher order boundary value problem. A curious feature of this method is that it is usually easier to employ it when the Green's function $G(x,y) \rightarrow \infty$ as $y$ approaches $x$. The type of arguments used in this paper are easier to prove when the Green's function $G(x,y)$ remains bounded as $y$ approaches $x$.
\end{remark}

\bibliographystyle{amsplain}

\end{document}